\documentclass[a4paper, 12pt, reqno]{amsart}

\usepackage[utf8]{inputenc}
\usepackage[T1,T2A]{fontenc}
\usepackage[english]{babel}
\usepackage{xcolor}
\usepackage{amscd, amsfonts, amsmath, amssymb, amsthm}
\usepackage[shortlabels]{enumitem}
\setlist[enumerate]{label=\normalfont{(\arabic*)}}
\usepackage[colorlinks=true, linkcolor=blue, citecolor=blue]{hyperref}
\usepackage{stmaryrd}
\usepackage{tikz-cd}
\usepackage[all]{xy}

\usepackage{euscript}
\usepackage{makecell}
\usepackage{xifthen}

\newcommand{\gen}[1]{\left\langle #1 \right\rangle}
\newcommand{\insertText}[2][]{%
    \ifthenelse{\isempty{#1}}{%
        \quad \text{#2} \quad%
    }{%
        #1 \text{#2} #1%
    }%
}
\newcommand{\pairing}[2]{\left\langle #1, #2 \right\rangle}

\DeclareMathOperator{\Aut}{Aut}
\DeclareMathOperator{\Cl}{Cl}

\DeclareMathOperator{\cone}{cone}
\DeclareMathOperator{\Eff}{Eff}
\DeclareMathOperator{\Hom}{Hom}
\DeclareMathOperator{\Ker}{Ker}

\DeclareMathOperator{\PSL}{PSL}
\DeclareMathOperator{\rk}{rk}
\DeclareMathOperator{\SAut}{SAut}
\DeclareMathOperator{\SL}{SL}
\DeclareMathOperator{\Spec}{Spec}

\renewcommand{\AA}{\mathbb{A}}
\renewcommand{\ge}{\geqslant}
\renewcommand{\le}{\leqslant}

\newcommand{\AAA}{\mathcal{A}}
\newcommand{\CCC}{\mathcal{C}}
\newcommand{\GG}{\mathbb{G}}
\newcommand{\Ga}{\mathbb{G}_\mathrm{a}}
\newcommand{\Gm}{\mathbb{G}_\mathrm{m}}
\newcommand{\HH}{\mathbb{H}}
\newcommand{\III}{\mathcal{I}}
\newcommand{\KK}{\mathbb{K}}

\newcommand{\NNN}{\mathcal{N}}
\newcommand{\OO}{\mathbb{O}}
\newcommand{\OOO}{\mathcal{O}}
\newcommand{\OOOO}{\mathfrak{O}}
\newcommand{\PP}{\mathbb{P}}
\newcommand{\QQ}{\mathbb{Q}}
\newcommand{\reg}{\mathrm{reg}}
\newcommand{\RRR}{\mathcal{R}}
\newcommand{\TT}{\mathbb{T}}
\newcommand{\VVV}{\mathcal{V}}
\newcommand{\WWW}{\mathcal{W}}
\newcommand{\ZZ}{\mathbb{Z}}

\theoremstyle{definition}
\newtheorem{definition}{Definition}
\newtheorem{example}{Example}

\theoremstyle{plain}
\newtheorem{corollary}{Corollary}
\newtheorem{lemma}{Lemma}
\newtheorem{problem}{Problem}
\newtheorem{proposition}{Proposition}
\newtheorem{theorem}{Theorem}

\theoremstyle{remark}
\newtheorem{remark}{Remark}

\begin{document}


\date{}
\title[Toric varieties and Gale duality]{Automorphisms of toric varieties and Gale duality}

\author{Ivan Arzhantsev}
\address{Faculty of Computer Science, HSE University, Pokrovsky Boulevard 11, Moscow, 109028 Russia}
\email{arjantsev@hse.ru}

\author{Kirill Shakhmatov}
\address{Faculty of Computer Science, HSE University, Pokrovsky Boulevard 11, Moscow, 109028 Russia}
\email{kshahmatov@hse.ru}

\thanks{The work is supported by the RSF grant 25-11-00302.}

\subjclass[2020]{Primary 14J50, 14M25; \ Secondary 20K21, 52B35}
\keywords{Toric variety, polyhedral fan, invariant divisor, automorphism, Demazure root}


\begin{abstract}
We classify complete toric threefolds $X$ such that the automorphism group $\Aut(X)$ acts on $X$ with an open orbit whose complement does not contain a divisor. The latter condition means that for any ray $\rho$ of the fan $\Sigma_X$ there is a Demazure root of $\Sigma_X$ associated with $\rho$. We also find among these varieties those $X$ for which the group $\Aut(X)$ is transitive on the smooth locus $X^\reg$. The classifications are based on Gale-dual interpretations of these properties.
\end{abstract}

\maketitle


\section{Introduction}
\label{sec1}

We consider algebraic varieties over an algebraically closed field $\KK$ of characteristic zero. We say that a variety $X$ is \emph{uniform} if the automorphism group $\Aut(X)$ acts on the smooth locus $X^\reg$ transitively. Clearly, the set $X^\reg$ is the maximal possible orbit of the group $\Aut(X)$ in $X$. The class of uniform varieties contains the class of \emph{homogeneous} varieties, i.e., algebraic varieties $X$ such that $\Aut(X)$ is transitive on $X$. The latter class includes homogeneous spaces $G / H$ of an algebraic group $G$, but in fact it is much broader.

Let us say that an algebraic variety $X$ is \emph{homogeneous in codimension $k$} if the group $\Aut(X)$ acts on $X$ with an open orbit $\OOOO$ and the codimension of the complement to $\OOOO$ in $X$ is at least $k + 1$. Any normal uniform variety is homogeneous in codimension one. Embeddings of homogeneous spaces with small boundary~\cite{AH} provide other examples of varieties that are homogeneous in codimension one. Further, if an algebraic group $G$ acts on a variety $X$ with an open orbit $\OOO$, then $X$ is homogeneous in codimension one if and only if any prime divisor $D$ in the complement to $\OOO$ can be moved by an automorphism of $X$ in such a way that the image of $D$ intersects $\OOO$.

The class of uniform varieties includes flexible quasiaffine varieties in the sense of~\cite{AFKKZ-1, FKZ}. In this case the group of special automorphisms $\SAut(X)$ acts on the smooth locus $X^{\reg}$ infinitely transitively. It follows from~\cite{AKZ} that any quasiaffine toric variety is uniform; see~\cite[Theorem~4.3]{ASZ} for details. Recently, it was proved that any non-degenerate affine spherical variety is flexible~\cite{Sh1} and so uniform; see~\cite{GSh, Sh} for earlier results.

It is known that the only homogeneous complete toric varieties are products of projective spaces~\cite[Theorem~3.9]{Ba}. The problem of classification of non-complete homogeneous toric varieties is much more complicated; see~\cite{Ar, AG}.

Notice that the automorphism group $\Aut(X)$ of a ``typical'' complete toric variety $X$ coincides with the acting torus $T$. This follows from the description of the group $\Aut(X)$ given in~\cite{Cox, De}. In particular, the open orbit $\OOOO$ of the group $\Aut(X)$ in this case coincides with the open $T$-orbit $\OOO$, and its complement contains divisors. 

The aim of this paper is to classify homogeneous in codimension one complete toric varieties up to dimension~3. We also find all uniform varieties among them. Since the only complete toric curve is the projective line $\PP^1$, we start with dimension~2.

\begin{proposition} \label{prpr}
Complete toric surfaces that are homogeneous in codimension one form the following list:
$$
\PP^1 \times \PP^1
\insertText{and}
\PP(1, 1, a), \ a \ge 1,
$$
where $\PP(a, b, c)$ is the weighted projective plane with weights $a, b, c$. All these surfaces are uniform.
\end{proposition}

The case of toric threefolds is more diverse. 

\begin{theorem} \label{thth} 
Complete toric threefolds that are homogeneous in codimension one form the following list:
$$
\PP(1, 1, a, b), \  1 \le a < b; \quad
\PP(a, a, b, b) \insertText[\ ]{and} \ \PP(a,a,b,b)[d], \
1 \le a \le b, \ (a, b) = 1, \ d \ge 2;
$$
$$
\PP^1 \times \PP(1, 1, a), \ 1 \le a; \quad
\PP^1 \times \PP^1 \times \PP^1;
$$
$$
Y(a, b), \ Y'(a,b), \ 1 \le a \le b; \quad
Y''(a, b), \ 1 \le a < b; 
$$
where $\PP(a,b,c,d)$ denotes a weighted projective space and $\PP(a,a,b,b)[d]$ is a fake weighted projective space; 
see Claims~1--2 below for notation. Among these varieties, non-uniform varieties are precisely
$\PP(1, 1, a, b)$ with $1 < a < b$, $(a, b)=1$ and
$Y'(1,b)$ with $b \ge 1$.
Moreover, $Y'(1, b)$, $b \ge 1$, which are the projective bundles $\PP(\OOO_{\PP^1}\oplus\OOO_{\PP^1}\oplus\OOO_{\PP^1}(b))$, 
are the only smooth non-uniform homogeneous in codimension one complete toric threefolds.
\end{theorem}

The case of the varieties $\PP(\OOO_{\PP^1}\oplus\OOO_{\PP^1}\oplus\OOO_{\PP^1}(b))$ motivates the following problem.

\begin{problem} \label{kuz}
Classify all smooth complete toric varieties that are homogeneous in codimension one.
\end{problem}

Notice that in Theorem~\ref{thth} the fake weighted projective spaces $\PP(a, a, b, b)[d], \ d \ge 2$, are the only varieties with torsion in the divisor class group, and $Y(a,b)$ are the only non-simplicial varieties. All varieties in Theorem~\ref{thth} are projective.

\smallskip

Our classification results are based on Gale duality. In Sections~\ref{sec2}-\ref{sec3} we recall basic facts on the Gale transform and its lattice version, respectively. Section~\ref{sec4} provides classification results in terms of Gale dual configurations for further use. In Section~\ref{sec5} we turn to toric geometry. We recall basic facts on toric varieties including the concepts of root subgroups in the automorphism group and Demazure roots of the associated fan. The main results concerning these concepts are summarized in Theorems~\ref{teo1} and~\ref{teo2}. For the convenience of the reader, we give short proofs of these results in terms of Cox rings in the Appendix. In Section~\ref{sec6}, for a toric variety, we interpret the properties of being uniform and homogeneous in codimension one in terms of Gale duality. In Section~\ref{sec7} we apply this interpretation and classifications from Section~\ref{sec4} to prove Proposition~\ref{prpr} and Theorem~\ref{thth}. In Section~\ref{sec8} we give a description of complete toric varieties with reductive automorphism group (Proposition~\ref{pred}). This property is important because of the problem of existence of a K\"ahler-Einstein metric; see~\cite{Nill}.

\smallskip

The authors are grateful to Alexander Kuznetsov for stimulating questions. In particular, Problem~\ref{kuz} was formulated as such a question.


\section{Gale transform and sparse configurations}
\label{sec2}


We begin with linear Gale duality. The presentation follows \cite[Section~2.2.1]{ADHL}; see also \cite{OP}, \cite[Section~6.4]{Zie}, \cite[Section~14.3]{CLS}, and~\cite{Pan} for a recent treatment. All vector spaces below are finite-dimensional vector spaces over the field of rational numbers $\QQ$. By a \emph{vector configuration} in a vector space $V$ we mean a finite multiset of vectors $\{v_1, \ldots, v_m\}$ in $V$ that spans the space $V$. A vector configuration $\VVV = \{v_1, \ldots, v_m\}$ in a vector space $V$ and a vector configuration $\WWW=\{w_1, \ldots, w_m\}$ in a vector space $W$ are called \emph{Gale dual} if the following conditions hold:
\begin{enumerate}
\item[(i)]
$v_1 \otimes w_1 + \ldots + v_m \otimes w_m = 0$ in
$V \otimes W$;


\item[(ii)]
for any vector space $U$ and any vectors $u_1, \ldots, u_m \in U$ with
$$
v_1 \otimes u_1 + \ldots + v_m \otimes u_m = 0
\insertText{in}
V \otimes U,
$$
there is a unique linear map $\psi \colon W \to U$ with $\psi(w_i) = u_i$ for $i = 1, \ldots, m$;


\item[(iii)]
for any vector space $U$ and any vectors $u_1, \ldots, u_m \in U$ with
$$
u_1 \otimes w_1 + \ldots + u_m \otimes w_m = 0
\insertText{in}
U \otimes W,
$$
there is a unique linear map $\phi \colon V \to U$ with $\phi(v_i) = u_i$ for $i = 1, \ldots, m$.
\end{enumerate}

If we fix one configuration in a Gale dual pair, then the other one is determined up to isomorphism. Therefore one configuration is called the \emph{Gale transform} of the other.

Consider vector configurations $\VVV = \{v_1, \ldots, v_m\}$ and $\WWW = \{w_1, \ldots, w_m\}$ in vector spaces $V$ and $W$ respectively, and let $V^*$ be the dual vector space of $V$. Then Gale duality of $\VVV$ and $\WWW$ is characterized by the following property: for any tuple $(\alpha_1, \ldots, \alpha_m)\in \QQ^m$ one has
$$
\alpha_1 w_1 + \ldots + \alpha_m w_m = 0
\iff
l(v_i) = \alpha_i
\insertText[\ \ ]{for} i = 1, \ldots, m
\insertText[\ \ ]{with some} l \in V^*.
$$

Let us present a construction which produces the Gale dual for a configuration $\VVV =$ \linebreak
$= \{v_1, \ldots, v_m\}$ in a space $V$. Take the vector space $\QQ^m$ and consider a surjective linear map $\alpha \colon \QQ^m \to V$ given on the standard basis $e_1, \ldots, e_m$ in $\QQ^m$ by $\alpha(e_i) = v_i$, $i = 1, \ldots, m$.

Consider two mutually dual short exact sequences of vector spaces
\begin{equation}
\begin{tikzcd}[trim left=(a), trim right=(a)]
0 \ar{r} & [0.75em] \Ker(\alpha) \ar{r} & [1.3em]
|[alias=a]| \QQ^m \ar{r}{\alpha} & [0.6em]
V \ar{r} & 0
\end{tikzcd}
\end{equation}
and
\begin{equation}
\begin{tikzcd}[trim left=(a), trim right=(a)]
0 & \big( \Ker(\alpha) \big)^* \ar{l} &
|[alias=a]| (\QQ^m)^* \ar[swap]{l}{\beta} &
V^* \ar{l} & 0 \ar{l}.
\end{tikzcd}
\end{equation}
Let $e_1^*, \ldots, e_m^*$ be the basis in $(\QQ^m)^*$ dual to the basis $e_1, \ldots, e_m$ in $\QQ^m$. Letting $W =$ \linebreak
$= \big( \Ker(\alpha) \big)^*$ and $w_i = \beta(e_i^*)$ for $i = 1, \ldots, m$, we obtain the Gale dual configuration $\WWW = \{w_1, \ldots, w_m\}$. In particular, we have $\dim V + \dim W = m$.

\medskip

Now let us apply linear Gale duality to some specific vector configurations.

\begin{definition}
A vector configuration $\VVV = \{v_1, \ldots, v_m\}$ in a vector space $V$ is called \emph{sparse} if for any $v_i \in \VVV$ there is a linear function $l_i \in V^*$ such that
$$
l_i(v_i) < 0 \insertText{and} l_i(v_j) \ge 0 \quad
\forall j \ne i.
$$
\end{definition}

Equivalently, any $v_i \in \VVV$ is not contained in $\cone(v_1, \ldots, \widehat{v}_i, \ldots, v_m)$, where for vectors $u_1, \ldots, u_s \in V$ we define
$$
\cone(u_1, \ldots, u_s) = \{
    \lambda_1 u_1 + \ldots + \lambda_s u_s \mid
    \lambda_i \in \QQ_{\ge 0} \ \forall i = 1, \dots, s
\}.
$$
The property of a configuration $\VVV$ to be sparse does not change if we multiply the vectors in $\VVV$ by positive scalars. In other words, sparseness is a property of a collection of rays generated by vectors in a vector configuration. Note that $\VVV$ is sparse if and only if for its Gale dual configuration $\WWW = \{w_1, \dots, w_m\}$ one has $w_i \in \cone(w_1, \dots, \widehat{w}_i, \dots, w_m)$ for all $w_i \in \WWW$.

\begin{definition}
A vector configuration $\VVV = \{v_1, \ldots, v_m\}$ in a vector space $V$ is called \emph{complete} if $\cone(v_1, \ldots, v_m) = V$.
\end{definition} 

Let us reformulate the completeness condition of a vector configuration $\VVV = \{v_1, \ldots, v_m\}$ in a vector space $V$ as follows: if for a function $l \in V^*$ one has $l(v_i) \ge 0$ for all $v_i \in \VVV$, then $l = 0$. Therefore, $\VVV$ is complete if and only if the implication
$$
\lambda_1 w_1 + \ldots + \lambda_m w_m = 0
\insertText[\ \ ]{for some}
\lambda_1, \dots, \lambda_m \in \QQ_{\ge 0}
\implies
\lambda_1 = \ldots = \lambda_m = 0
$$
holds for its Gale dual configuration $\WWW = \{w_1, \dots, w_m\}$. The latter is equivalent to saying that $w_i \ne 0$ for all $w_i \in \WWW$ and $\cone(w_1, \ldots, w_m)$ is strictly convex. We call such configurations $\WWW$ \emph{strictly convex}.

\begin{lemma} \label{lem11}
A complete configuration $\VVV$ is sparse if and only if any one-dimensional face of $\cone(w_1, \ldots, w_m)$ contains at least two elements of $\WWW$.
\end{lemma}

\begin{proof}
Denote
$$
\sigma = \cone(w_1, \ldots, w_m)
\insertText{and}
\sigma_i = \cone(w_1, \ldots, \widehat{w}_i, \ldots, w_m)
$$
for each $i = 1, \dots, m$. As we have seen above, a configuration $\VVV$ is sparse if and only if $w_i \in \sigma_i$ for all $i = 1, \dots, m$. Since $\VVV$ is assumed to be complete, all the vectors $w_i$ are nonzero.

Assume that $\VVV$ is sparse and consider a vector $w_i \in \WWW$ lying on a one-dimensional face of~$\sigma$. From $w_i \in \sigma_i$ it follows that there exists a vector $w_j \in \WWW$, $j \ne i$, which lies on the same one-dimensional face. Conversely, if any one-dimensional face of $\sigma$ contains at least two elements of $\WWW$, then $\sigma_i = \sigma$ for all $i = 1, \dots, m$, and the vector $w_i$ is contained in~$\sigma_i$.
\end{proof}

\begin{remark} \label{rho} 
Lemma~\ref{lem11} implies that if $\dim V = n$ then any complete sparse vector configuration in $V$ contains $\le 2n$ vectors.
\end{remark}

Let us say that two vector configurations $\VVV=\{v_1, \ldots, v_m\}$ and $\VVV'=\{v_1', \ldots, v_m'\}$ in $V$ are \emph{equivalent} if there is an automorphism of the space $V$ that sends, up to renumbering, the ray generated by $v_i$ to the ray generated by $v_i'$ for all $i=1, \ldots, m$.

For small $n$, one can easily find all complete sparse vector configurations up to equivalence. For example, with $n = 2$ we have only
\begin{equation} \label{condim2}
\VVV_1 = \big\{ (1, 0), (0, 1), (-1, -1) \big\}
\insertText{and}
\VVV_2 = \big\{ (1, 0), (0, 1), (-1, 0), (0, -1) \big\}.
\end{equation}
The Gale transforms of these configurations are the multisets $\WWW_1 = \{ f_1^{(3)} \}$ with $\dim W = 1$ and
$\WWW_2 = \{ f_1^{(2)}, f_2^{(2)} \}$ with $\dim W = 2$, respectively. Here we denote by $f_1, \ldots, f_{m - n}$ an appropriate basis in $W$ and the notation $a^{(k)}$ means that an element $a$ is contained in a multiset $\AAA$ with multiplicity $k$.

\begin{proposition} \label{3DCSC.pr}
Up to equivalence, complete sparse vector configurations in a three-dimensional vector space are precisely the configurations
$$
\VVV_1 = \VVV_0 \cup \{(-1, -1, -1)\}, \quad
\VVV_2 = \VVV_0 \cup \{(-1, -1, 0), (0, 0, -1)\},
$$
$$
\VVV_3 = \VVV_0 \cup \{(-1, -1, 0), (0, -1, -1)\}, \quad \VVV_4 = \VVV_0 \cup \{(-1, 0, 0), (0, -1, 0), (0, 0, -1)\},
$$
where $\VVV_0 = \{ (1, 0, 0), (0, 1, 0), (0, 0, 1) \}$.
\end{proposition}

\begin{proof}
We use Lemma~\ref{lem11} and classify dual configurations. It is clear that the only possibilities are
$$
\WWW_1 = \{ f_1^{(4)} \}
\insertText{with}
\dim W = 1,
$$
$$
\WWW_2 = \{ f_1^{(3)}, f_2^{(2)} \}
\insertText{and}
\WWW_3 = \{ f_1^{(2)}, f_2^{(2)}, f_1 + f_2 \}
\insertText{with}
\dim W = 2,
$$
and
$$
\WWW_4 = \{ f_1^{(2)}, f_2^{(2)}, f_3^{(2)} \}
\insertText{with}
\dim W = 3.
$$
\end{proof}

Starting from $\dim V = 4$, we have infinitely many non-equivalent complete sparse vector configurations in $V$. For example, one may take vector configurations dual to
$$
\WWW = \{
    f_1^{(2)}, \ f_2^{(2)}, \ f_1 + f_2, \ f_1 + \lambda f_2
\}
\insertText{with}
\dim W = 2 \insertText{and} \lambda \ge 1.  
$$


\section{Lattice Gale transform}
\label{sec3}


Now we come to a version of Gale duality that takes into account the lattice of integer points in a rational vector space. This construction is introduced in~\cite[Section~5]{Ar}.

Given an abelian group $A$, we denote by $A_\QQ$ the vector space $A \otimes_\ZZ \QQ$ over $\QQ$. By a \emph{lattice} we mean a free finitely generated abelian group.

A \emph{vector configuration} $\NNN$ in a lattice $N$ is a finite multiset of vectors $p_1, \ldots, p_m \in N$ that spans the vector space $N_{\QQ}$. Note that $\NNN$ gives rise to a vector configuration 
$$
\NNN_\QQ = \{ p_1 \otimes 1, \dots, p_m \otimes 1 \}
$$ 
in the vector space $N_{\QQ}$.

Let $\NNN = \{p_1, \ldots, p_m\}$ be a vector configuration in a lattice $N$. Consider the lattice $\ZZ^m$ with the standard basis $e_1, \ldots, e_m$ and an exact sequence of lattices
\begin{equation} \label{seq3}
\begin{tikzcd}
0 \ar{r} & L \ar{r} &
\ZZ^m \ar{r}{\alpha} &
N
\end{tikzcd}
\end{equation}
defined by $\alpha(e_i) = p_i$, $i = 1, \ldots, m$. We identify the dual lattice $\Hom(\ZZ^m, \ZZ)$ with $\ZZ^m$ using the dual basis $e_1^*, \ldots, e_m^*$. Denote $M = \Hom(N, \ZZ)$ and let $N \times M \to \ZZ$, $(p, u) \mapsto \pairing{p}{u}$ be the canonical pairing. The homomorphism $M \to \ZZ^m$ dual to $\alpha$ gives rise to a short exact sequence of abelian groups
\begin{equation} \label{seq4}
\hspace{4.3em}
\begin{tikzcd}
0 & A \ar{l} &
\ZZ^m \ar[swap]{l}{\beta} &
M \ar{l} & 0 \ar{l}.
\end{tikzcd}
\end{equation}
Let $a_i = \beta(e_i^*)$. By construction, the elements $a_1, \ldots, a_m$ generate the abelian group~$A$. We call the multiset $\AAA = \{a_1, \ldots, a_m\}$ the \emph{lattice Gale transform} of the configuration~$\NNN$. By tensoring sequences~(\ref{seq3}) and~(\ref{seq4}) with $\QQ$ we obtain linear Gale duality between $\VVV=\NNN_\QQ$ and $\WWW=\AAA_\QQ$.

Conversely, given a multiset $\{a_1, \ldots, a_m\}$ that generates an abelian group $A$, we can reconstruct sequence~(\ref{seq4}), the lattice $N = \Hom(M, \ZZ)$, the dual homomorphism $\alpha \colon \ZZ^m \to N$, and the vectors $p_1, \ldots, p_m$ in $N$. This shows that the lattice Gale transform defines a duality between vector configurations in lattices and finite generating multisets in abelian groups. However, the symmetry present in linear Gale duality is lost: we have a lattice $N$ on the one side and an arbitrary finitely generated abelian group $A$ on the other side. Note that $\rk A = m - \rk N$ and the abelian group $A$ is a lattice if and only if the vectors $p_1,\ldots, p_m$ generate the lattice~$N$. 

\begin{remark} \label{remus}
Let $\NNN = \{p_1, \ldots, p_m\}$ be a vector configuration in a lattice $N$ and $\AAA =$ \linebreak
$=\{a_1, \dots, a_m\}$ be its lattice Gale transform in an abelian group $A$. The vector $p_1$ is a primitive vector in $N$ if and only if the elements $a_2, \ldots, a_m$ generate the group $A$. Indeed, $p_1$ is primitive if and only if there is an element $e \in M$ such that $\pairing{p_1}{e} = 1$, or, equivalently, there is a relation $a_1 + \alpha_2 a_2 +\ldots + \alpha_m a_m = 0$ with some integers $\alpha_2, \ldots, \alpha_m$. More generally, a subset $\{p_i, i \in I\}$ in $\NNN$ can be extended to a basis of $N$ if and only if the subset $\{a_j, j\notin I\}$ generates the group $A$. 
\end{remark}

\begin{example}
The lattice Gale transform of the configuration $\NNN = \{p_1, p_2\}$ in $N = \ZZ^2$ with $p_1 = (1, 0)$ and $p_2 = (1, 2)$ is the multiset
$\AAA = \{a_1, a_2\} = \{ \bar{1}^{(2)} \}$
in the group $A = \ZZ / 2 \ZZ$. On the other hand, the linear Gale transform of the configuration $\VVV = \{v_1, v_2\}$ in $V = \QQ^2$ with $v_1 = (1, 0)$ and $v_2 = (1, 2)$ is the multiset $\WWW = \{ 0^{(2)} \}$ in the space $W = \{0\}$.
\end{example}

From now on we assume that
\begin{equation} \label{con1}
\text{the configuration $\NNN$ consists of pairwise distinct primitive vectors in $N$.}
\end{equation}

Let us say that a pair $(a_i, a_j)$ of elements in $\AAA$ with $i\ne j$ is \emph{extremal} if $a_i - a_j$ lies in the subgroup generated by $\AAA \setminus \{a_i, a_j\}$ and this subgroup has infinite index in $A$. Similarly, a pair $(w_i, w_j)$ of elements in $\WWW$ with $i\ne j$ is \emph{extremal} if the elements $\WWW \setminus \{w_i, w_j\}$ span a hyperplane in $W$ and the vector $w_i - w_j$ lies in this hyperplane.

We say that a multiset $\AAA$ in $A$ is \emph{proper} if for all $a_i \in \AAA$ the multiset $\AAA \setminus \{a_i\}$ generates the group $A$ and for all $a_i, a_j \in \AAA$ with $i\ne j$ either the multiset $\AAA \setminus \{a_i ,a_j\}$ generates a subgroup of finite index in $A$ or the pair $(a_i, a_j)$ is extremal. This definition implies that for all $i\ne j$ either the vectors $\WWW \setminus \{w_i, w_j\}$ span $W$ or the pair $(w_i, w_j)$ is extremal.

The condition dual to~(\ref{con1}) is
\begin{equation} \label{con2}
\text{the multiset $\AAA$ in $A$ is proper.}
\end{equation}
Indeed, the first assertion in the definition of a proper multiset means that all vectors in $\NNN$ are primitive, and the second assertion reflects the fact that for $i \ne j$ the elements $p_i, p_j \in \NNN$ are either non-proportional or opposite.

\smallskip

In what follows, all multisets $\AAA$ are assumed to be proper. We call a vector configuration $\NNN$ in a lattice $N$ \emph{complete} if $\NNN_\QQ$ is a complete vector configuration in $N_\QQ$. As before, the dual condition for the Gale dual multiset $\AAA = \{a_1, \dots, a_m\}$ is the implication
$$
\alpha_1 a_1 + \ldots + \alpha_m a_m = 0 \
\text{ for some } \
\alpha_1, \dots, \alpha_m \in \ZZ_{\ge 0}
\implies
\alpha_1 = \ldots = \alpha_m = 0.
$$
In this case we say that $\AAA$ is \emph{strictly convex}. By tensoring sequences~(\ref{seq3}) and~(\ref{seq4}) with $\QQ$ we see that a generating multiset $\AAA$ in an abelian group $A$ is strictly convex if and only if the configuration $\AAA_\QQ$ in $A_\QQ$ is strictly convex.


\section{Suitable configurations and admissible multisets}
\label{sec4}


We keep the notation of the previous sections.

\begin{definition}
A vector configuration $\NNN$ in a lattice $N$ is called \emph{suitable} if for each $p \in \NNN$ there exists $u \in M$ such that
$$
\pairing{p}{u} = -1
\insertText{and}
\pairing{p'}{u} \ge 0 \quad
\forall p' \in \NNN \setminus \{p\}.
$$
\end{definition}

Note that a suitable configuration $N$ consists of pairwise distinct primitive vectors. Moreover, the corresponding vector configuration in $N_{\QQ}$ is sparse.

\begin{definition}
A finite generating multiset $\AAA$ in an abelian group $A$ is called \emph{admissible} if each element $a \in \AAA$ is contained in the monoid generated by $\AAA \setminus \{a\}$.
\end{definition}

Clearly, a vector configuration $\NNN$ in a lattice $N$ is suitable if and only if its lattice Gale transform $\AAA$ is admissible; cf.~\cite[Lemma 4]{Ar}.

\smallskip

The aim of this section is to classify complete suitable configurations in lattices of rank~2 and~3. We begin with a technical result that will be used later. For a finitely generated abelian group $A$ we fix a basis $\epsilon_1, \ldots, \epsilon_r$ of a free part of $A$. We use the same notation $\epsilon_1, \ldots, \epsilon_r$ for the basis $\epsilon_1 \otimes 1, \ldots, \epsilon_r \otimes 1$ in the vector space $W = A_{\QQ}$.

\begin{proposition} \label{SEC.le}
Let $\AAA$ be a strictly convex admissible multiset in an abelian group $A$ of rank~$r$. Assume that the number $m$ of elements in $\AAA$ does not exceed $2 r + 1$. Then either $m = 2 r$ or $m = 2 r + 1$, and there is  a number $1 \le l \le r$ such that the vector configuration $\AAA_\QQ$ is equivalent to 
$$
\text{either} \quad
\WWW = \{ \epsilon_1^{(2)}, \dots, \epsilon_r^{(2)} \}
\insertText{or}
\WWW' = \WWW \cup \{ \epsilon_1 + \ldots + \epsilon_l \}.
$$
Moreover, $A$ is a lattice and $\AAA$ is isomorphic to 
$$
\text{either} \quad
\{ \epsilon_1^{(2)}, \dots, \epsilon_r^{(2)} \}
\insertText{or}
\{
    \epsilon_1^{(2)}, \dots, \epsilon_r^{(2)},
    \gamma_1 \epsilon_1 + \ldots + \gamma_l \epsilon_l
\}
$$
respectively, for some positive integers $\gamma_1, \dots, \gamma_l$.
\end{proposition}

\begin{proof}
The statement on $\AAA_\QQ$ follows from Lemma~\ref{lem11}. In the case when $A$ is known to be a lattice, the multiset $\AAA$ has the form
$\{ v_1^{(2)}, \ldots, v_r^{(2)}, v \}$,
where
$v = \gamma_1 v_1 + \ldots + \gamma_l v_l$
with some positive rational numbers $\gamma_1, \ldots, \gamma_l$. Since the multiset $\AAA \setminus \{v\}$ generates the lattice~$A$, the vectors $v_1, \ldots, v_r$ form a basis in $A$. So we may assume that $v_i = \epsilon_i$ for all $i$. In this case the coefficients $\gamma_1, \ldots, \gamma_l$ are integers, and the assertion follows.

\smallskip

Assume now that $A = \ZZ^r \oplus B$, where $B$ is a non-trivial finite abelian group. The next lemma follows directly from definitions.

\begin{lemma} \label{lemproj}
Let $\AAA$ be a strictly convex admissible multiset in an abelian group $A = C \oplus B$, where $B$ is finite. Then the projection of $\AAA$ to $C$ along $B$ is a strictly convex admissible multiset in $C$.
\end{lemma}

So we may assume that the projection of $\AAA$ to $\ZZ^r$ along $B$ has the desired form. Then we have
$$
\AAA = \big\{
    (\epsilon_1, b_1), \ (\epsilon_1, b_1'), \ \ldots,
    (\epsilon_r, b_r), \ (\epsilon_r, b_r'), \
    (\gamma_1 \epsilon_1 + \ldots + \gamma_l \epsilon_l, b)
\big\}
$$
with some $b_1 ,b_1', \ldots, b_r, b_r', b \in B$. Let us show that $b_i = b_i'$ for all $i = 1, \dots, r$. Let $i = 1$ for simplicity. Since $\AAA$ is admissible, there exist non-negative integers
$
\alpha, \alpha_2, \alpha_2', \dots,
\alpha_r, \alpha_r', \beta
$
such that
$$
(\epsilon_1, b_1) =
\alpha (\epsilon_1, b_1') +
\alpha_2 (\epsilon_2, b_2) + \alpha_2' (\epsilon_2, b_2') +
\ldots
$$
$$
\ldots +
\alpha_r (\epsilon_r, b_r) + \alpha_r' (\epsilon_r, b_r') +
\beta (\gamma_1 \epsilon_1 + \ldots + \gamma_l \epsilon_l, b).
$$
Multiplying this equality by a large enough positive integer $s$, we obtain an equality
$$
s (\epsilon_1, 0) =
s \alpha (\epsilon_1, 0) +
s \alpha_2 (\epsilon_2, 0) + s \alpha_2' (\epsilon_2, 0) +
\ldots
$$
$$
\ldots +
s \alpha_r (\epsilon_r, 0) + s \alpha_r' (\epsilon_r, 0) +
s \beta
(\gamma_1 \epsilon_1 + \ldots + \gamma_l \epsilon_l, 0).
$$
If $\alpha > 0$ then the equality
$$
0 =
s (\alpha - 1) (\epsilon_1, 0) +
s \alpha_2 (\epsilon_2, 0) + s \alpha_2' (\epsilon_2, 0) +
\ldots
$$
$$
\ldots +
s \alpha_r (\epsilon_r, 0) + s \alpha_r' (\epsilon_r, 0) +
s \beta
(\gamma_1 \epsilon_1 + \ldots + \gamma_l \epsilon_l, 0)
$$
implies that $\alpha = 1$ and
$
\alpha_2 = \alpha_2' = \ldots =
\alpha_r = \alpha_r' = \beta = 0
$
since $\AAA$ is strictly convex. Therefore, $b_1 = b_1'$. Assume that $\alpha = 0$. Then $\beta = \gamma_1 = 1$,
$
\alpha_2 = \alpha_2' = \ldots = \alpha_r = \alpha_r' = 0
$
and $l = 1$, hence $b_1 = b$. A symmetric argument shows that $b_1' = b$. We conclude that $b_i = b_i'$ for all $i = 1, \dots, r$.

Fix a nonzero element $b_0 \in B$. Since the multiset
$
\AAA \setminus \{
    (\gamma_1 \epsilon_1 + \ldots + \gamma_l \epsilon_l, b)
\}
$
generates the group $A$, the element $(0, b_0)$ is an integer linear combination of elements of this multiset. It follows that
$$
x_1 \epsilon_1 + \ldots + x_r \epsilon_r = 0
\insertText{and}
x_1 b_1 + \ldots + x_r b_r = b_0
$$
for some integers $x_1, \ldots, x_r$. But the first equality implies $x_1 = \ldots = x_r = 0$, a contradiction with the second equality. So the assumption that $B$ is non-trivial leads to a contradiction, and Proposition~\ref{SEC.le} is proved.
\end{proof}

We come to a classification of complete suitable configurations in a lattice of rank~2.

\begin{proposition} \label{use} 
Complete suitable vector configurations in $\ZZ^2$ up to automorphism are
$$
\NNN_1^a = \{ (1, 0), (-1, -a), (0, 1) \}
\insertText[\ ]{with} a \ge 1
\insertText[\ \ ]{and}
\NNN_2 = \{ (1, 0), (-1, 0), (0, 1), (0, -1) \}.
$$
The corresponding lattice Gale transforms are
$$
\AAA_1^a = \{ 1^{(2)}, a \}
\insertText[\ ]{in} \ZZ
\insertText{and}
\AAA_2 = \{ (1, 0)^{(2)}, \ (0, 1)^{(2)} \}
\insertText[\ ]{in} \ZZ^2.
$$
\end{proposition}

\begin{proof}
Clearly, $\AAA_1^a$ is the lattice Gale transform of $\NNN_1^a$ and $\AAA_2$ is the lattice Gale transform of $\NNN_2$. Therefore, it suffices to show that any strictly convex admissible multiset $\AAA$ with $r$ elements in an abelian group $A$ of rank $r - 2$ is isomorphic to either $\AAA_1^a$ or $\AAA_2$. By~(\ref{condim2}) the vector configuration $\AAA_\QQ$ is isomorphic to one of the configurations $\WWW_i$ in $W = A_\QQ$:
$$
\WWW_1 = \{ \epsilon_1^{(3)} \} =
\{ \epsilon_1^{(2)}, \epsilon_1 \}
\insertText[\ ]{with} \dim W = 1, \quad
\WWW_2 = \{\epsilon_1^{(2)}, \epsilon_2^{(2)}\}
\insertText[\ ]{with} \dim W = 2.
$$
Now the claim follows from Proposition~\ref{SEC.le}.
\end{proof}

We proceed with a classification of admissible multisets in a lattice of rank~3.

\begin{theorem} \label{3DCAM.th}
The list of strictly convex admissible multisets $\AAA$ with $m$ elements in an abelian group $A$ of rank $m - 3$ up to isomorphism is
\begin{enumerate}
\item
$m = 4$ and $A = \ZZ$,
$\AAA = \{ 1^{(2)}, a, b \}$, where
$1 \le a < b$;
\smallskip

\item
$m = 4$ and $A = \ZZ$,
$\AAA = \{ a^{(2)}, b^{(2)} \}$, where
$1 \le a \le b$ and $(a, b) = 1$;
\smallskip

\item
$m = 4$ and $A = \ZZ \oplus \ZZ / d \ZZ$,
$\AAA = \{ (a, \bar{u})^{(2)}, \ (b, \bar{v})^{(2)} \}$, where
$d \ge 2$, $1 \le a \le b$, $(a, b) = 1$, and
$a \bar{v} - b \bar{u} = \bar{1}$;
\smallskip

\item
$m = 5$ and $A = \ZZ^2$,
$\AAA = \{ (1, 0)^{(2)}, \ (0, 1)^{(2)}, \ (a, 0 ) \}$, where
$a \ge 1$;
\smallskip

\item
$m = 5$ and $A = \ZZ^2$,
$\AAA = \{ (1, 0)^{(2)}, \ (0, 1)^{(2)}, \ (a, b) \}$, where
$1 \le a \le b$;
\smallskip

\item
$m = 6$ and $A = \ZZ^3$,
$
\AAA = \{
    (1, 0, 0)^{(2)}, \ (0, 1, 0)^{(2)}, \ (0, 0, 1)^{(2)}
\}.
$
\end{enumerate}
Moreover, multisets from different items \textnormal{(1)-(6)} or from one item but with different $(a, b, d)$ are pairwise non-isomorphic. Multisets from \textnormal{(3)} with the same $(a, b, d)$, but with different $(\bar{u}, \bar{v})$ are isomorphic.
\end{theorem}

\begin{proof}
Since the multiset $\AAA$ is strictly convex, the rank of the group $A$ is positive and we have $m \ge 4$. Lemma~\ref{lem11} implies
$m \le 2 n = 2 (m - (m - 3)) = 6$.
The cases $m = 5, 6$ follow from Proposition~\ref{SEC.le}. So from now on we assume that $m = 4$. Here $\AAA_\QQ$ is isomorphic to a configuration
$$
\WWW_1 = \{ e_1^{(4)} \} \insertText{with} \dim W = 1.
$$

\emph{Case~1. }\
Let $A = \ZZ$. After possibly replacing the generator of $\ZZ$ by its negative, all elements of $\AAA$ are positive. Since $\AAA$ is admissible it follows that 
$\AAA$ contains its minimal element with multiplicity at least two. So we have $\AAA = \{ a^{(2)}, b, c \}$ for some $1 \le a \le b \le c$.

If $a = 1$, then there is no restriction on $b, c$ and we are done. Let $a > 1$. If $b \ne c$, then $b = a k$ for some $k \ge 1$, so $(a, c) = 1$ and therefore
$c \notin \gen{a, b}_{\ZZ_{\ge 0}}$,
a contradiction. Hence $b = c$ and $(a, b) = 1$. This concludes Case~1.

\smallskip

Now assume that the group $A$ has torsion. Then
$
A =
\ZZ \oplus \ZZ / d_1 \ZZ \oplus \ldots \oplus \ZZ / d_s \ZZ,
$
where $d_i \ge 2$ and $d_i$ divides $d_{i + 1}$ for $1 \le i \le s - 1$. Since the group $A$ is generated by $\AAA \setminus \{a\}$ for any $a \in \AAA$, it is generated by at most three elements, so $s \le 2$. Hence it remains to consider two more cases.

\smallskip

\emph{Case~2. }\
Let $A = \ZZ \oplus \ZZ / d \ZZ$ with $d \in \ZZ_{\ge 2}$. By Lemma~\ref{lemproj}, we assume that the projection of $\AAA$ to $\ZZ$ along $\ZZ / d \ZZ$ is as in Case~1.

First we assume that
$$
\AAA = \{
    (1, \bar{t}_1), \ (1, \bar{t}_2), \
    (a, \bar{t}_3), \ (b, \bar{t}_4)
\}
$$
for some $1 \le a < b$. Since the element $(1, \bar{t}_2)$ lies in the monoid generated by all other elements of $\AAA$, we have either $\bar{t}_2 = \bar{t}_1$, or $\bar{t}_2 = \bar{t}_3$ and $a = 1$. By the symmetric argument, we conclude that $\bar{t}_1 = \bar{t}_2 = \bar{t}$. Then we have $\bar{t}_3 = a \bar{t}$ and $\bar{t}_4 = b \bar{t}$. Let us show that in this case the multiset does not generate the group $A$.  More precisely, the element $(0, \bar{1})$ is not in the group generated by $\AAA$. Indeed, the system
$$
\left\{
\begin{array}{ccc}
x \cdot 1 + y \cdot 1 + z \cdot a + u \cdot b & = & 0 \\
x \cdot \bar{t} + y \cdot \bar{t} + z \cdot a \bar{t} +
u \cdot b \bar{t} & = & \bar{1}
\end{array}
\right.
$$
has no solution in $\ZZ$, since we have
$
(x + y + z \cdot a + u \cdot b) \bar{t} =
0 \bar{t} \ne \bar{1}.
$

Now we assume 
$$
\AAA = \{
(a, \bar{t}_1), \
(a, \bar{t}_2), \
(b, \bar{t}_3), \
(b, \bar{t}_4)
\}
$$
for some $1 \le a \le b$ and $(a, b) = 1$. Similarly, we obtain $\bar{t}_1 = \bar{t}_2 = \bar{u}$. For $\bar{t}_3$ we have either $\bar{t}_3 = \bar{t}_4$ or $\bar{t}_3 = b \bar{u}$ and $a = 1$. But the second variant leads to $\bar{t}_4 = b \bar{u}=\bar{t}_3$ . So we have
$$
\AAA = \{ (a, \bar{u})^{(2)}, \ (b, \bar{v})^{(2)} \}.
$$
We need to check when $\AAA$ generates the group $A$. It suffices to know that $(1, \bar{0})$ and $(0, \bar{1})$ are in the subgroup generated by $\AAA$.

Assume that we have
$x (a, \bar{u}) + y (b, \bar{v}) = (0, \bar{1})$
for some integers $x$ and $y$. Since $(a, b)=1$, there are integers $g$ and $h$ such that $a g + b h = 1$. Then
$g (a, \bar{u}) + h (b, \bar{v}) = (1, \bar{w})$
for some $w \in \ZZ$ and
$$
(g - w x) (a, \bar{u}) + (h - w y) (b, \bar{v}) = (1, \bar{0}).
$$
These arguments show that it suffices to know when $(0, \bar{1})$ is in the group generated by $\AAA$.

Again assume
$x (a, \bar{u}) + y (b, \bar{v}) = (0, \bar{1})$
for some integers $x$ and $y$. Then $x = -k b$ and $y = k a$ for some integer $k$, so
$-k b \bar{u} + k a \bar{v} = \bar{1}$.
The latter equation has an integer solution~$k$ if and only if $(a v - b u, d) = 1$. This completes the consideration of Case~2.

\smallskip

\emph{Case~3. }\
Let
$A = \ZZ \oplus \ZZ / d_1 \ZZ \oplus \ZZ / d_2 \ZZ$
with $d_1, d_2 \ge 2$ and $d_1$ divides $d_2$. Taking the projections of $\AAA$ to $\ZZ \oplus \ZZ / d_1 \ZZ$ along $\ZZ / d_2 \ZZ$ and to $\ZZ \oplus \ZZ / d_2 \ZZ$ along $\ZZ / d_1 \ZZ$ and applying Lemma~\ref{lemproj} and Case~2, we conclude that
$$
\AAA = \{
    (a, \bar{u}_1, \bar{u}_2)^{(2)}, \
    (b, \bar{v}_1, \bar{v}_2)^{(2)}
\},
$$
where $1 \le a \le b$, $(a, b) = 1$,
$(a v_1 - b u_1, d_1) = 1$, and
$(a v_2 - b u_2, d_2) = 1$.
Since $\AAA$ generates the group $A$, there are integers $x$ and $y$ such that 
$$
x (a, \bar{u}_1, \bar{u}_2) + y (b, \bar{v}_1, \bar{v}_2) =
(0, \bar{0}, \bar{1}).
$$
The equality of the first coordinates implies $x = -k b$, $y = k a$ for some integer $k$. For the third coordinates we have
$-k b u_2 + k a v_2 = 1 + s d_2$
for some integer $s$. This shows that $(k, d_2) = 1$. In particular, $(k, d_1) = 1$. For the second coordinates we have 
$$
-k b u_1 + k a v_1 = l d_1
$$
for some integer $l$. This contradicts the facts that
$(k, d_1) = 1$ and
$(a v_1 - b u_1, d_1) = 1$.
So Case~3 is not possible.

\smallskip

It remains to check whether the obtained multisets are pairwise non-isomorphic. This is straightforward if $A$ is a lattice. If $A = \ZZ \oplus \ZZ / d \ZZ$, we reconstruct $a$ and $b$ by projecting the multiset to $\ZZ$ along the torsion subgroup, while $d$ is the order of the torsion subgroup. Applying an automorphism of $\ZZ / d \ZZ$ we can replace the condition
$(a v - b u, d) = 1$
with the condition
$a \bar{v} - b \bar{u} = \bar{1}$,
hence $a v - b u = 1 + l d$ for some integer~$l$. Since $a$ and $b$ are coprime, there exist integers $u_0$ and $v_0$ with
$a v_0 - b u_0 = 1$.
Then
$u = u_0 (1 + l d) + s a$ and
$v = v_0 (1 + l d) + s b$
for some integer $s$. We conclude that
$\bar{u} = \bar{u}_0 + a \bar{s}$ and
$\bar{v} = \bar{v}_0 + b \bar{s}$.

Replace the generating set $(1, \bar{0})$, $(0, \bar{1})$ of the group $A$ by $(1, -\bar{s})$, $(0, \bar{1})$. This transition defines an automorphism of $A$ that sends $(a, \bar{u})$ to $(a, \bar{u} + a \bar{s})$ and $(b, \bar{v})$ to $(b, \bar{v} + b \bar{s})$. This shows that the multiset $\AAA$ corresponding to a pair $(\bar{u}, \bar{v})$ is isomorphic to the multiset corresponding to the fixed pair $(\bar{u}_0, \bar{v}_0)$. So the multiset $\AAA$ is independent, up to isomorphism, of the choice of $(\bar{u}, \bar{v})$. 

This completes the proof of Theorem~\ref{3DCAM.th}.
\end{proof}

Direct computations of Gale duals to the multisets listed in Theorem~\ref{3DCAM.th} lead to the following result.

\begin{corollary} \label{3DCSC.co}
The list of complete suitable vector configurations $\NNN$ in $\ZZ^3$, where each vector in $\NNN$ is written as a column of a matrix of size $3 \times |\NNN|$, up to automorphism of $\ZZ^3$ is the following:

\medskip 

\begin{center}
\normalfont
\begin{tabular}{rlrlrl}
(1) 
&
$
\begin{pmatrix}
1 & 0 & 0 & -1 \\
0 & 1 & 0 & -a \\
0 & 0 & 1 & -b
\end{pmatrix}
$
&


(3)
&
$
\begin{pmatrix}
1 & 0 &   -1 &    0 \\
0 & 1 &    0 &   -1 \\
0 & 0 & ad & -bd
\end{pmatrix}
$
&


(5)
&
$
\begin{pmatrix}
1 & 0 & 0 & -1 &  0 \\
0 & 1 & 0 &  0 & -1 \\
0 & 0 & 1 & -a & -b
\end{pmatrix}
$
\medskip
\\


(2)
&
$
\begin{pmatrix}
1 & 0 & -1 &  0 \\
0 & 1 &  0 & -1 \\
0 & 0 &  a & -b
\end{pmatrix}
$
&


(4)
&
$
\begin{pmatrix}
1 & 0 & 0 & -1 &  0 \\
0 & 1 & 0 &  0 & -1 \\
0 & 0 & 1 &  0 & -a
\end{pmatrix}
$
&


(6)
&
$
\begin{pmatrix}
1 & 0 & 0 & -1 &  0 &  0 \\
0 & 1 & 0 &  0 & -1 &  0 \\
0 & 0 & 1 &  0 &  0 & -1
\end{pmatrix},
$
\end{tabular}
\end{center}

\smallskip

\noindent where $1 \le a < b$ in \textnormal{(1)},
$1 \le a \le b$ and $(a, b) = 1$ in \textnormal{(2)},
$1 \le a \le b$ and $(a, b) = 1$
with $d \ge 2$ in \textnormal{(3)},
$a \ge 1$ in \textnormal{(4)}, and
$1\le a\le b$ in \textnormal{(5)}.
\end{corollary}


\section{Toric varieties, Demazure roots, and root subgroups}
\label{sec5}


We start with a brief reminder on the theory of toric varieties; see~\cite{Fu, CLS} for details. We work over an algebraically closed field $\KK$ of characteristic zero. Denote by $\Gm$ and $\Ga$ the multiplicative group and the additive group of the ground field, respectively, considered as one-dimensional linear algebraic groups.

By a \emph{torus} we mean a linear algebraic group $T$ isomorphic to $\Gm^n$ for some non-negative integer $n$, called the \emph{rank} of $T$. A normal irreducible variety $X$ is called \emph{toric} if there is a faithful action of a torus $T$ on $X$ with an open orbit.

Fix a torus $T$ of rank $n$. Let $N$ be the lattice of one-parameter subgroups of $T$ and ${M = \Hom_\ZZ(N, \ZZ)}$ be the character lattice. For any $p\in N$ let $\lambda_p \colon \Gm \to T$ be the corresponding one-parameter subgroup and for any $u \in M$ let $\chi^u \colon T \to \Gm$ be the corresponding character. There is a natural pairing $N \times M \to \ZZ$, $(p, u) \mapsto \pairing{p}{u}$ defined by $\chi^u \big( \lambda_p(t) \big) = t^{\pairing{p}{u}}$.
Denote by $N_\QQ$ and $M_\QQ$ the associated $\QQ$-vector spaces $N \otimes_\ZZ \QQ$ and $M \otimes_\ZZ \QQ$. Let 
$\pairing{\cdot}{\cdot} \colon N_\QQ \times M_\QQ \to \QQ$ denote the extension of this pairing.

Recall that the \emph{dual cone} $\sigma^\vee$ to a cone $\sigma$ in $N$ is defined by
$$
\sigma^\vee = \{
    u \in M_\QQ \mid \pairing{v}{u} \geq 0 \
    \forall v \in \sigma
\}.
$$
Given a strictly convex cone $\sigma$ in $N$ one can consider the finitely generated $\KK$-algebra $\KK[\sigma^\vee \cap M]$ graded by the monoid of lattice points of the cone $\sigma^\vee$:
$$
\KK[\sigma^\vee \cap M] =
\bigoplus_{u \in \sigma^\vee \cap M} \KK \chi^u,
$$
where $\chi^u \cdot \chi^{u'} = \chi^{u + u'}$. Denote by $X_\sigma$ the affine variety $\Spec(\KK[\sigma^\vee \cap M])$. It is well-known that given an affine variety $X$ there is a bijection between faithful $T$-actions on $X$ and effective $M$-gradings on $\KK[X]$. The aforementioned $(\sigma^\vee \cap M)$-grading on $\KK[\sigma^\vee \cap M]$ corresponds to a faithful $T$-action on $X_\sigma$ with an open orbit. Therefore, $X_\sigma$ is a toric variety. Moreover, every affine toric variety $X$ occurs this way.

A \emph{fan} in $N$ is a finite collection $\Sigma$ of strictly convex cones in $N$ such that for all $\sigma_1, \sigma_2 \in \Sigma$ every face of $\sigma_1$ is an element of $\Sigma$ and the intersection $\sigma_1 \cap \sigma_2$ is a face of both $\sigma_1$ and~$\sigma_2$. There is a one-to-one correspondence between toric varieties with an acting torus $T$ and fans in~$N$. Namely, given a fan $\Sigma$ in $N$ the corresponding toric variety $X_\Sigma$ is a union of open affine charts $X_\sigma$, $\sigma \in \Sigma$, where any two charts $X_{\sigma_1}$ and $X_{\sigma_2}$ are glued along their common open subset $X_{\sigma_1 \cap \sigma_2}$. Every toric variety arises this way.

There is a bijection between $T$-orbits on the toric variety $X_\Sigma$ and cones $\sigma \in \Sigma$; see \cite[Section~3.1]{Fu} and \cite[Section~3.2]{CLS}. We denote by $O(\sigma)$ the $T$-orbit in $X$ corresponding to a cone $\sigma \in \Sigma$. In particular, the cone $\sigma = \{0\}$ corresponds to the open $T$-orbit $\OOO$ in $X$, and cones $\sigma$ of dimension $n$ represent $T$-fixed points on $X$. Further, the closures $D_1, \ldots, D_m$ of the orbits $O(\rho_1), \ldots, O(\rho_m)$,
where $\rho_1, \ldots, \rho_m$ are the rays of $\Sigma$, form the set of all $T$-invariant prime divisors on~$X$. A point $x \in O(\sigma)$ is contained in the divisor $D_i$ if and only if the ray $\rho_i$ is contained in the cone $\sigma$.

Recall that a cone $\sigma$ is \emph{regular} if the set of primitive vectors on rays of $\sigma$ can be extended to a basis of $N$. A point $x \in O(\sigma)$ is a regular point of the variety $X$ if and only if the cone $\sigma$ is regular.

A toric variety $X$ is \emph{non-degenerate} if $\KK[X]^\times = \KK^{\times}$ or, equivalently, the fan $\Sigma$ is not contained in a proper subspace in $N_\QQ$. One more
equivalent condition is that there is no decomposition $T = T' \times T_0$, where $T'$ and $T_0$ are subtori of positive dimension and $X = X' \times T_0$, where $X'$ is a $T'$-toric variety; see~\cite[Proposition 3.3.9]{CLS}.

Denote by $|\Sigma|$ the \emph{support} of a fan $\Sigma$ in $N$:
$$
|\Sigma| = \bigcup_{\sigma \in \Sigma} \sigma.
$$
A fan $\Sigma$ in $N$ is called \textit{complete} if $|\Sigma| = N_\QQ$. The corresponding toric variety $X_\Sigma$ is complete if and only if the fan $\Sigma$ is complete.

\smallskip

Now we come to basic results on root subgroups in the automorphism group $\Aut(X)$ of a toric variety $X$. If $\Ga \times X \to X$ is a nontrivial regular action on an algebraic variety~$X$, then the image $H$ of the group $\Ga$ in $\Aut(X)$ is called a \emph{$\Ga$-subgroup}. If the variety $X$ is toric with the acting torus $T$, then a $\Ga$-subgroup of $\Aut(X)$ normalized by $T$ is called a \emph{$T$-root subgroup} or just a \emph{root subgroup}. Since $t H(a) t^{-1} \in H$ for all $t \in T$ and $a \in \Ga$, we have $t H(a) t^{-1} = H(\chi^e(t) a)$ for some $e = e(H) \in M$. Here $H(a)$ denotes the image of $a$ under the map $\Ga \to \Aut(X)$.

Let $X$ be a non-degenerate toric variety and $\Sigma$ be the corresponding fan in the lattice $N$. Let $\Sigma(1) = \{ \rho_1, \ldots, \rho_m \}$ be the set of rays of the fan $\Sigma$. Denote by $p_i$ the primitive lattice vector on the ray $\rho_i$. We also use the notation $p_{\rho}$ for the primitive lattice vector on a ray $\rho$.

\begin{definition} \label{defdr}
A vector $e \in M$ is called a \emph{Demazure root} of a fan $\Sigma$ in the lattice $N$ if there is an index $1 \le i \le m$ such that the following two conditions hold:
\begin{enumerate}
\item[(R1)]
we have $\pairing{p_i}{e} = -1$ and $\pairing{p_j}{e} \ge 0$ for all $j \ne i$;


\item[(R2)]
if $\sigma \in \Sigma$ and $e|_{\sigma} \equiv 0$, then $\cone(\sigma, \rho_i)$ is in $\Sigma$ as well. 
\end{enumerate} 
\end{definition} 

\begin{remark}
Condition (R2) holds automatically if the fan $\Sigma$ has convex support.
\end{remark}

We say that the ray $\rho_i$ from Definition~\ref{defdr} is \emph{associated} with the Demazure root $e$. Sometimes we denote this ray by $\rho_e$. Let $\RRR_{\rho}(\Sigma)$ denote the set of Demazure roots of a fan $\Sigma$ associated with a ray $\rho$. For the set $\RRR(\Sigma)$ of all Demazure roots,
we have
$$
\RRR(\Sigma) =
\bigsqcup_{\rho \in \Sigma(1)} \RRR_{\rho}(\Sigma).
$$

\begin{theorem} \label{teo1}
Let $X$ be a non-degenerate toric variety and $\Sigma$ be the corresponding fan. The map $H \to e(H)$ establishes  
a bijection between root subgroups in $\Aut(X)$ and Demazure roots of the fan $\Sigma$.
\end{theorem}

We denote by $H_e$ the root subgroup in $\Aut(X)$ corresponding to a root $e \in \RRR(\Sigma)$. Let $\lambda_e$ be the one-parameter subgroup in $T$ corresponding to the primitive vector $p_e \in N$ on the ray $\rho_e$. The next theorem describes the action of the subgroup $H_e$ on $X$.

\begin{theorem} \label{teo2}
Let $X$ be a non-degenerate toric variety and $\Sigma$ be the corresponding fan.
\begin{enumerate}
\item 
For any root subgroup $H_e$ in $\Aut(X)$ and any $T$-orbit $O(\sigma)$ on $X$ with $\sigma \in \Sigma$, either $O(\sigma)$ consists of $H_e$-fixed points, or 
every $H_e$-orbit intersecting $O(\sigma)$ meets exactly one other $T$-orbit.
\item
In the second case, $O(\sigma)$ belongs to a pair $\big( O(\sigma'), O(\sigma'') \big)$ of $T$-orbits such that $e|_{\sigma'} \equiv 0$ and $\sigma'' = \cone(\sigma', \rho_e)$, and each $H_e$-orbit in $O(\sigma') \cup O(\sigma'')$ intersects $O(\sigma')$ in a one-dimensional orbit of the subtorus $\lambda_e$, while the same $H_e$-orbit meets $O(\sigma'')$ in a single point.
\end{enumerate}
\end{theorem}

Theorems~\ref{teo1} and \ref{teo2} are well-known and can be found in one form or another in \cite[Th\'eor\`eme~3]{De}, \cite[Proposition~3.14]{Oda}, \cite[Section~4]{Cox}, and subsequent publications. For the convenience of the reader, we provide short proofs of these results in terms of Cox rings in the Appendix below.


\section{Combinatorial characterizations}
\label{sec6}


The aim of this section is to characterize uniform toric varieties and toric varieties that are homogeneous in codimension one in combinatorial terms. We also  apply the results of Section~\ref{sec4} to obtain classification results in dimensions~2 and~3.

Let $X$ be a normal algebraic variety and $X^{\reg}$ be the smooth locus of $X$. By $\Eff(X)$ we denote the monoid in $\Cl(X)$ generated by classes of effective Weil divisors on $X$.  For a point $x \in X$, let $\Gamma(x)$ be the monoid generated by classes of effective divisors whose support does not contain~$x$.

\begin{lemma} \label{lemma1}
Let $X$ be a non-degenerate toric variety and $x \in X^\reg$. Consider an effective divisor $D$ on $X$ whose support does not contain $x$. Then there is an effective $T$-invariant divisor $D'$ on $X$, which is linearly equivalent to $D$ and whose support does not contain $x$.
\end{lemma}

\begin{proof}
It suffices to prove the statement for $X = X^{\reg}$, and it is done, e.g., in \cite[Lemma~6]{Ar}.
\end{proof}

Since a toric variety $X$ contains an open $T$-orbit $\OOO$, the automorphism group $\Aut(X)$ acts on $X$ with an open orbit $\OOOO$ which contains the orbit $\OOO$. By Lemma~\ref{lemma1}, if $x \in O(\sigma)$ for some regular cone $\sigma \in \Sigma$, then the monoid $\Gamma(x)$ is generated by classes of $T$-invariant prime divisors $D_{\rho}$, where $\rho \notin \sigma(1)$. In particular, $\Gamma(x) = \Eff(X)$ for any $x \in \OOO$.

\begin{lemma} \label{lemma2}
Let $\rho \in \Sigma(1)$. Then the class $[D_\rho]$ is in the monoid generated by other classes $[D_{\rho_1}], \ldots, [D_{\rho_s}]$ if and only if there is a vector $e \in M$ with
$$
\pairing{p_\rho}{e} = -1, \quad
\pairing{p_{\rho_i}}{e} \ge 0, \ i = 1, \ldots, s,
\insertText{and}
\pairing{p_{\rho'}}{e} = 0 \insertText{for all}
\rho' \ne \rho, \rho_1, \ldots \rho_s.
$$
\end{lemma} 

\begin{proof}
By \cite[Section~3.4]{Fu} or \cite[Theorem~4.1.3]{CLS} all relations between the classes $[D_i]$ in $\Cl(X)$ have the form
$\sum_i \pairing{p_i}{u} [D_i] = 0$,
where $u\in M$. A class $[D_\rho]$ is in the monoid generated by classes $[D_{\rho_1}], \ldots, [D_{\rho_s}]$ if and only if $[D_\rho] = \sum_{i = 1}^s \alpha_i [D_{\rho_i}]$
for some non-negative integers $\alpha_i$. Existence of such a relation is equivalent to existence of the desired vector~${e \in M}$.
\end{proof}

\begin{theorem} \label{main}
Let $X$ be a non-degenerate toric variety whose fan has convex support. Then the open orbit $\OOOO$ of the group $\Aut(X)$ is the set of points $x \in X^{\reg}$ with $\Gamma(x) = \Eff(X)$.
\end{theorem}

\begin{proof}
Clearly, the set of points $x \in X$ such that $\Gamma(x) = \Eff(X)$ is $\Aut(X)$-invariant. Since we have $\Gamma(x) = \Eff(X)$ for any point $x \in \OOO$, we conclude that $\Gamma(x) = \Eff(X)$ for any point $x \in \OOOO$. Conversely, let $\Gamma(x) = \Eff(X)$ for some $x \in O(\sigma)$, where $\sigma \in \Sigma$ is a regular cone. Let us show by induction on $\dim(\sigma)$ that $x \in \OOOO$. The base of induction corresponds to the case $\sigma = \{0\}$, i.e. $O(\sigma) = \OOO$, and we know that $\OOO \subseteq \OOOO$. Now assume that $\sigma \ne \{0\}$ and $\rho \in \sigma(1)$.

It follows from Lemma~\ref{lemma1} that the class $[D_\rho]$ equals
$\sum_i \alpha_{\rho_i} [D_{\rho_i}]$,
where $\rho_i \notin \sigma(1)$ and $\alpha_{\rho_i}$ are non-negative integers. By Lemma~\ref{lemma2} there is a root $e \in \RRR_\rho(\Sigma)$ such that $\sigma = \cone(\rho, \sigma')$ with $\sigma' \in \Sigma$ and $e|_{\sigma'} \equiv 0$. From Theorem~\ref{teo2} we conclude that the subset $O(\sigma) \cup O(\sigma')$ is contained in an $\Aut(X)$-orbit. In particular, $\Gamma(x) = \Eff(X)$ for all $x \in O(\sigma')$. By inductive hypothesis, the orbit $O(\sigma')$ is contained in $\OOOO$. Therefore, $O(\sigma) \subseteq \OOOO$.
\end{proof}

For a non-degenerate toric variety $X$, we have two dual exact sequences 
\begin{equation}
\begin{tikzcd}
0 \ar{r} & L \ar{r} &
\ZZ^m \ar{r}{\alpha} & N
\end{tikzcd}
\end{equation}
and
\begin{equation}
\hspace{4.3em}
\begin{tikzcd}
0  & A \ar{l} &
\ZZ^m \ar[swap]{l}{\beta} & M \ar{l} & 0 \ar{l},
\end{tikzcd}
\end{equation}
where $A = \Cl(X)$. Let us denote by $p(\Sigma)$ the set $\{p_1, \ldots, p_m\}$ of primitive lattice vectors on the rays of the fan $\Sigma$. The Gale dual configuration $\AAA =\AAA(\Sigma) = \{a_1, \ldots, a_m\}$ in $A$ to the configuration $p(\Sigma)$ in $N$ is the configuration of classes $[D_1], \ldots, [D_m]$ of $T$-invariant prime divisors in $\Cl(X)$. The monoid $\Eff(X)$ is identified with the monoid $A_+$ generated by $\AAA$, and $\Gamma(x)$ for a point $x \in O(\sigma)$ is the monoid $\Gamma(\sigma)$ generated by $a_s$ with $p_s \notin \sigma$.

\begin{corollary} \label{uniform}
Let $X$ be a non-degenerate toric variety whose fan has convex support. Then $X$ is uniform if and only if for any cone $\sigma\in\Sigma$ the condition that $\Gamma(\sigma)$ generates the group~$A$ implies $\Gamma(\sigma) = A_+$.
\end{corollary}

\begin{proof}
Since $X^\reg = \bigcup O(\sigma)$, where $\sigma$ runs through all regular cones $\sigma \in \Sigma$, Theorem~\ref{main} implies that the variety $X$ is uniform if and only if $\Gamma(\sigma) = A_+$ for all regular cones $\sigma \in \Sigma$.

\begin{lemma} \label{lreg}
A cone $\sigma \in \Sigma$ is regular if and only if the monoid $\Gamma(\sigma)$ generates the group~$A$.
\end{lemma}

\begin{proof}
A cone $\sigma$ is regular if and only if for any $p_i \in \sigma$ there is $u \in M$ such that $\pairing{p_i}{u} = 1$ and $\pairing{p_j}{u} = 0$ for all $p_i \ne p_j \in \sigma$. This is equivalent to the fact that any $a_i$ with $p_i \in \sigma$ is contained in the subgroup generated by all $a_s$ with $p_s \notin \sigma$.
\end{proof}

This completes the proof of Corollary~\ref{uniform}.
\end{proof}

\begin{remark}
We observe that a toric variety $X$ is smooth if and only if for any $\sigma \in \Sigma$ the monoid $\Gamma(\sigma)$ generates the group $A$.
\end{remark}

The following result is essentially obtained in~\cite[Theorem~2.1]{AKZ}; see also \cite[Theorem~4.3]{ASZ}. Here we give a short proof which uses the technique developed in this paper.

\begin{proposition}
Any quasi-affine toric variety $X$ is uniform.
\end{proposition}

\begin{proof}
We may assume that $X$ is non-degenerate. Again it suffices to show that any non-open orbit $O(\sigma)$ with a regular $\sigma \in \Sigma$ is connected by some root subgroup $H_e$ with an orbit $O(\sigma')$ of higher dimension. Take a facet $\sigma'$ of $\sigma$ with $\sigma = \cone(\rho, \sigma')$. 
There is a vector $e' \in M$ with $\pairing{p_\rho}{e'} = -1$ and $e'|_{\sigma'} \equiv 0$.

Let us recall that $X$ is quasiaffine if and only if $\Sigma$ is a subfan of some cone $\widehat{\sigma}$; see, e.g.,~\cite[Theorem~1.6]{PV}. So we can find an element $l \in M$ which is non-negative on $\widehat{\sigma}$ and vanishes precisely on $\sigma$ in $\widehat{\sigma}$. Letting $e = e' + tl$ with $t\gg 0$ we have 
$\pairing{p_\rho}{e} = -1$,
$\pairing{p_{\rho'}}{e} > 0$
for $\rho' \notin \sigma$, and $e|_{\sigma'} \equiv 0$. Since the cone $\sigma = \cone(\rho, \sigma')$ has convex support, condition (R2) is fulfilled for~$e$ and the connecting root subgroup $H_e$ exists.
\end{proof}

\begin{remark}
One can check that a toric variety $X$ is quasiaffine if and only if for any cone $\sigma\in\Sigma$ the cone generated by vectors $a_i\otimes 1$, where $p_i\notin\sigma$, is a subspace of $A_{\QQ}$. 
\end{remark}

\begin{remark}
Let $\sigma''\in\Sigma$ and $\sigma'$ be a facet of $\sigma''$. The orbits $O(\sigma')$ and $O(\sigma'')$ are connected by some root subgroup $H_e$ in the sense of Theorem~\ref{teo2}.(2) if and only if $\Gamma(\sigma')=\Gamma(\sigma'')$, cf.~\cite[Lemma~3.3]{Ba}: this property holds both for toric varieties
whose fan $\Sigma$ has convex support and for quasiaffine toric varieties; see~\cite[Proposition~2.7]{AB} for the affine version. 
\end{remark}

\begin{proposition} \label{codim2}
Let $X$ be a non-degenerate toric variety whose fan has convex support. Then $X$ is homogeneous in codimension one if and only if the configuration $p(\Sigma)$ in $N$ is suitable. 
\end{proposition}

\begin{proof} 
The variety $X$ is homogeneous in codimension one if and only if the open orbit $\OOOO$ of $\Aut(X)$ intersects all $T$-invariant prime divisors on $X$. If $p(\Sigma)$ is suitable, then for any ray $\rho\in\Sigma(1)$ there is a Demazure root $e\in\RRR_{\rho}(\Sigma)$, and the root subgroup $H_e$ connects the open $T$-orbit $O(\rho)$ in the divisor $D_{\rho}$ with $\OOO$, so the orbit $\OOOO$ intersects $D_{\rho}$. 

Conversely, if $\OOOO$ intersects $D_{\rho}$, then $\Gamma(x)=\Eff(X)$ for a general point $x\in D_{\rho}$. By Lemma~\ref{lemma1}, the class $[D_{\rho}]$ is in the monoid generated by other classes of $T$-invariant prime divisors. By Lemma~\ref{lemma2}  this implies that there is a root $e\in\RRR_{\rho}(\Sigma)$. We conclude that the configuration $p(\Sigma)$ is suitable. 
\end{proof}

\begin{remark} 
It follows from Remark~\ref{rho} that for any complete homogeneous in codimension one  toric variety $X$ the rank of the divisor class group
$\Cl(X)$ does not exceed $\dim X$, and by Proposition~\ref{SEC.le} the equality holds only for $X=(\PP^1)^n$.  At the same time, any affine toric variety $X$ is homogeneous in codimension one, but starting from dimension three one cannot bound the rank of $\Cl(X)$ by a function in $\dim X$.
\end{remark}

\begin{remark}
Let us show that the assumption that the fan $\Sigma$ has convex support is essential. Let $X$ be obtained from $\PP^1\times\PP^1$ by removing four $T$-fixed points. This corresponds to removing cones of maximal dimension from the fan of $\PP^1\times\PP^1$.  It can be deduced from \cite[Theorem~4.2.4.1]{ADHL}  that the group $\Aut(X)$ is the group of automorphisms of $\PP^1\times\PP^1$ that preserve the set of $T$-fixed points. It is easy to see that this group is a finite extension of the acting torus $T$. We conclude that the variety $X$ is not homogeneous in codimension one, and Theorem~\ref{main}, Corollary~\ref{uniform}, and Proposition~\ref{codim2} do not hold for $X$.  
\end{remark} 

Finally, let us show that the results of this section can be applied to degenerate toric varieties as well. 

\begin{proposition}
Let $X = X_\Sigma$ be a degenerate toric variety whose fan has convex support. Consider a toric decomposition $X = X' \times T_0$ with $T = T' \times T_0$, where $X'$ is a non-degenerate toric $T'$-variety. Then $X$ is homogeneous in codimension one or uniform if and only if $X'$~is.
\end{proposition}

\begin{proof}
Notice that the open orbit $\OOOO$ of the group $\Aut(X)$ on $X$ contains the subset $\OOOO'\times T_0$, where $\OOOO'$ is the open orbit of $\Aut(X')$ on $X'$. It suffices to prove that $\OOOO=\OOOO'\times T_0$.

Since $\OOOO$ is $T_0$-invariant, it has a form $U\times T_0$, where $U$ is a subset of $X'$ containing~$\OOOO'$. Then for any 
$x=(x',t_0)\in U\times T_0$ the monoid of classes of effective divisors on $X$ whose support does not contain $x$ coincides with the monoid of all effective classes on $X$.  

We claim that Lemma~\ref{lemma1} holds for degenerate toric varieties as well. Indeed, its proof in \cite[Lemma~6]{Ar} is based on linearization of line bundles from~\cite[Theorem~4.2.2.5]{ADHL},  which works also for varieties with nonconstant invertible global functions. So any effective divisor on $X$ whose support does not contain $x$ is linearly equivalent to a $T$-invariant effective divisor on $X$ with the same property. The latter are precisely the divisors of the form $D'\times T_0$, where $D'$ is an effective $T'$-invariant divisor on $X'$ and the support of $D'$ does not contain~$x'$. Taking into account the natural identification $\Cl(X'\times T_0)\cong\Cl(X')$, we conclude that $\Gamma(x')=\Eff(X')$ on the non-degenerate toric variety $X'$. By Theorem~\ref{main}, we obtain $x'\in\OOOO'$, and so $\OOOO=\OOOO'\times T_0$. 
\end{proof}


\section{Proofs of main results}
\label{sec7}


We are ready to prove the results formulated in the Introduction. 

\begin{proof}[Proof of Proposition~\ref{prpr}] 
The assertion follows from Proposition~\ref{codim2}, Proposition~\ref{use}, and Corollary~\ref{uniform}. 
\end{proof} 

\begin{proof}[Proof of Theorem~\ref{thth}] 
By Proposition~\ref{codim2}, we should describe complete fans $\Sigma$ of toric varieties $X$ whose configurations $p(\Sigma)$ are suitable. The corresponding configurations are given as columns of matrices (1)-(6) in Corollary~\ref{3DCSC.co}.  Using the technique of bunches of cones from~\cite{BH}, one easily shows the following.

\smallskip

{\it Claim 1.}\ In Cases (1)-(4) and (6), there is exactly one complete fan whose primitive ray generators are the columns of the corresponding matrices. Cases~(1) and~(2) correspond to weighted projective 3-spaces $\PP(1,1,a,b)$ with $1\le a<b$ and $\PP(a,a,b,b)$ with $1\le a\le b$, $(a,b)=1$, respectively; see \cite[Section~2.2]{Fu}. Case~(4) represents $\PP^1\times\PP(1,1,a)$ with $1\le a$ and Case~(6) represents $\PP^1\times\PP^1\times\PP^1$. Case~(3) gives a variety $\PP(a,a,b,b)[d]$ with $1\le a\le b$, $(a,b)=1$, and $d\ge 2$, which is obtained from $\PP(a,a,b,b)$ as a quotient by an action of the cyclic group~${\ZZ/d\ZZ}$. Such varieties are called fake weighted projective spaces; see~\cite[Exercise~5.1.13 and Remark~15.4.3]{CLS} or \cite{Kas}. Notice that $\PP(a,a,b,b)[d]$ are the only complete toric threefolds that are homogeneous in codimension one and have torsion in the divisor class group. 

\smallskip

{\it Claim 2.}\ Case~(5) corresponds to three toric varieties $Y'(a,b)$, $Y''(a,b)$, and $Y(a,b)$ with $1\le a\le b$. They are given by complete fans $\Sigma'(a,b)$, $\Sigma''(a,b)$, and $\Sigma(a,b)$, respectively. 

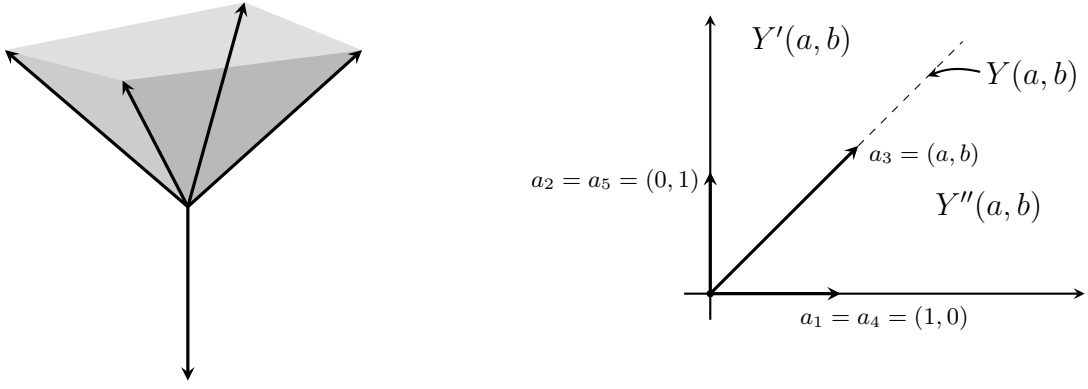
\begin{figure}[ht]

\centering

\begin{tikzpicture}[scale=1.15, >=stealth]

\begin{scope}[xshift=0cm, yshift=0cm]

\coordinate (O)  at (0, 0);
\coordinate (L)  at (-2.1, 1.8);
\coordinate (M)  at (-0.75, 1.45);
\coordinate (T)  at (0.65, 2.35);
\coordinate (R)  at (2.0, 1.8);

\fill[gray!25] (L) -- (T) -- (R) -- (M) -- cycle;
\fill[gray!40] (O) -- (L) -- (M) -- cycle;
\fill[gray!55] (O) -- (M) -- (R) -- cycle;

\draw[->, very thick] (O) -- (L);
\draw[->, very thick] (O) -- (M);
\draw[->, very thick] (O) -- (T);
\draw[->, very thick] (O) -- (R);
\draw[->, very thick] (O) -- (0, -2.0);

\end{scope}


\begin{scope}[xshift=6cm, yshift=-1cm]

\draw[->, thick] (-0.3, 0) -- (4.3, 0);
\draw[->, thick] (0, -0.3) -- (0, 3.2);

\fill (0, 0) circle (1.2pt);

\draw[->, very thick] (0, 0) -- (1.5, 0);
\node[below] at (2.0, -0.02)
{\scriptsize $a_1 = a_4 = (1, 0)$};

\draw[->, very thick] (0, 0) -- (0, 1.4);
\node[left, align=center] at (0.1, 1.3)
{\scriptsize
    $a_2 = a_5 = (0, 1)$
};

\draw[->, very thick] (0, 0) -- (1.7, 1.7);
\node[right] at (1.7, 1.6) {\scriptsize $a_3 = (a, b)$};

\draw[dashed] (1.7, 1.7) -- (2.9, 2.9);

\node at (1.05, 2.9) {$Y'(a, b)$};
\node at (3.7, 2.5) {$Y(a, b)$};
\node at (3.2, 1.0) {$Y''(a, b)$};

\draw[->, thick, bend right=18] (3.1, 2.55) to (2.5, 2.5);
\end{scope}
\end{tikzpicture}
\caption{Gale dual configurations for varieties $Y'(a,b)$, $Y''(a,b)$, and $Y(a,b)$.}
\label{fig}
\end{figure}

If we enumerate the columns of matrix~(5) as $p_1,\ldots, p_5$, then the set of maximal cones of each of these three fans includes
$$
\cone(p_1, p_2, p_3), \quad \cone(p_1, p_3, p_5), \quad \cone(p_2, p_3, p_4), \quad \cone(p_3, p_4, p_5). 
$$ 
The set of maximal cones of $\Sigma'(a,b)$ also includes 
$$
\cone(p_1, p_2, p_4) \quad \text{and} \quad \cone(p_1, p_4, p_5),
$$
the set of maximal cones of $\Sigma''(a,b)$ also includes  
$$
\cone(p_1, p _2, p_5) \quad \text{and} \quad \cone(p_2, p_4, p_5), 
$$
and the set of maximal cones of $\Sigma(a,b)$ also includes $\cone(p_1, p_2, p_4, p_5)$. In particular,  $Y(a,b)$ are the only non-simplicial complete toric threefolds which are homogeneous in codimension~one. It also follows that the varieties $Y'(a,b)$ and $Y''(a,b)$ are isomorphic if and only if $a=b$. 

By \cite[Corollary~9.3]{BH}, all varieties in Theorem~\ref{thth} are projective. The claim concerning non-uniform cases follows from Theorem~\ref{3DCAM.th} and Corollary~\ref{uniform}. The last claim about smoothness can be checked using Lemma~\ref{lreg}.   
\end{proof} 

\begin{example}
Let us take a closer look at the varieties $Y'(1,b)$, $Y''(1,b)$, and $Y(1,b)$. The authors are grateful to Constantin Shramov for discussing this example. 
The variety $Y'(1,b)$ is isomorphic to the projective bundle $\PP(\OOO_{\PP^1}\oplus\OOO_{\PP^1}\oplus\OOO_{\PP^1}(b))$ with projection $\pi'\colon Y'(1,b)\to\PP^1$; see~\cite[Example~7.3.5]{CLS}. The variety $Y''(1,b)$ is isomorphic to $Y'(1,b)$ if $b=1$ and it is singular otherwise. It also admits a surjection $\pi''\colon Y''(1,b)\to\PP^1$.  Since the fans of $Y'(1,b)$ and $Y''(1,b)$ are obtained from the two possible subdivisions of the non-simplicial cone in the fan of $Y(1,b)$, we get the following flip diagram: 
$$
\begin{tikzcd}
 &
Y'(1, b) \ar[swap]{dl}{\pi'} \ar{dr}{\varphi'} & &
Y''(1, b) \ar[swap]{dl}{\varphi''} \ar{dr}{\pi''} &
\\
\PP^1 & & Y(1, b) & & \PP^1
\end{tikzcd}
$$
cf.~\cite[Exercise~3.3.12]{CLS}. The variety $Y(1,b)$ contains a unique singular point $S$. We have
$$
\PP^1 \cong \PP \big( \OOO_{\PP^1}(b) \big) \cong
(\varphi')^{-1}(S) \cong (\varphi'')^{-1}(S),
$$
and the morphisms $\pi'$ and $\pi''$ restricted to these $\PP^1$'s are isomorphisms. Moreover, the maps $\varphi'\colon Y'(1,b)\setminus(\varphi')^{-1}(S)\to Y(1,b)\setminus S$ and $\varphi''\colon Y''(1,b)\setminus(\varphi'')^{-1}(S)\to Y(1,b)\setminus S$ are isomorphisms.

One can check that for $X=Y'(1,b)$ there are three monoids $\Gamma(\sigma_1)=\Gamma(\sigma_2)=\Gamma(\sigma_3)$ different from $\Eff(X)$. Using~\cite[Theorem~3.7]{Ba}, we see that $X$ consists of two $\Aut(X)$-orbits: the open orbit $\OOOO$ and its complement $(\varphi')^{-1}(S)\cong\PP^1$.  
\end{example}


\section{Complete toric varieties with reductive automorphism group}
\label{sec8}


In this section we describe complete toric varieties whose automorphism group is a reductive linear algebraic group, i.e., its unipotent radical is trivial. 
It turns out that this property has a nice interpretation in terms of Gale duality. 

We say that a linear algebraic group $G$ is \emph{locally isomorphic} to $\SL(m_1)\times\ldots\times\SL(m_k)$ if there is a finite central subgroup $F$ in 
$\SL(m_1)\times\ldots\times\SL(m_k)$ such that 
$$
G\cong(\SL(m_1)\times\ldots\times\SL(m_k))/F. 
$$
In other words, $G$ is an intermediate group between ${\SL(m_1)\times\ldots\times\SL(m_k)}$ and \linebreak
$\PSL(m_1)\times\ldots\times\PSL(m_k)$.

With any multiset $\AAA=\{a_1^{(m_1)},\ldots,a_k^{(m_k)}\}$ in an abelian group $A$, where $a_1,\ldots,a_k$ are pairwise distinct, one associates the subset $V$ in $A$ which consists of the same elements $a_1,\ldots,a_k$, but with multiplicity one each.

\begin{definition}
A subset $V$ in an abelian group $A$ is called \emph{minimal} if for any $a\in V$ the monoid in $A$ generated by $V\setminus\{a\}$ does not contain $a$. 
\end{definition}

Denote by $V(\Sigma)$ the subset associated with the multiset $\AAA(\Sigma)$ corresponding to a toric variety $X$ with fan $\Sigma$. 

\begin{proposition} \label{pred}
The automorphism group $\Aut(X)$ of a complete toric variety $X$ is reductive if and only if the subset $V(\Sigma)$ is minimal. In this case, the semisimple part of the group $\Aut(X)$ is locally isomorphic to
$$
\SL(m_1)\times\ldots\times\SL(m_k), 
$$
where $m_i$ are the multiplicities of elements in the multiset $\AAA(\Sigma )$. 
\end{proposition}

\begin{proof}
Let us recall that a Demazure root $e\in\RRR(\Sigma)$ is \emph{semisimple} if $-e\in\RRR(\Sigma)$. It is well known that the automorphism group $\Aut(X)$ is reductive if and only if all roots in $\RRR(\Sigma)$ are semisimple; see~\cite{De} and~\cite[Proposition~3.2]{Nill}. 

As we know, Demazure roots are given by expressions $a_i=\sum_{j\ne i} \alpha_ja_j$, where $a_j\in\AAA(\Sigma)$ and $\alpha_j$ are non-negative integers. Such a relation corresponds to a semisimple root if and only if it has the form $a_i=a_j$, and the first assertion follows. 

It also follows that semisimple roots are divided into $k$ parts. The $i$-th part forms a root system of type $A_{m_i-1}$. Therefore, the semisimple part of 
$\Aut(X)$ is locally isomorphic to $\SL(m_1)\times\ldots\times\SL(m_k)$.
\end{proof} 

\begin{remark}
Even if the automorphism group $\Aut(X)$ of a complete toric variety $X$ is not reductive, it follows from the proof of Proposition~\ref{pred} that the semisimple part of the linear algebraic group $\Aut(X)$ is locally isomorphic to $\SL(m_1)\times\ldots\times\SL(m_k)$, where $m_i$ are the multiplicities of elements in the multiset $\AAA(\Sigma)$. This result was originally proved by Demazure~\cite[Proposition~3.3]{De}.  
\end{remark} 

Let us proceed with a simple example. In the case $X=\PP^n$  we have $A=\ZZ$ and ${\AAA(\Sigma)=\{1^{(n+1)}\}}$. The automorphism group $\Aut(\PP^n)$ is the reductive group $\PSL(n+1)$.

Using Theorem~\ref{3DCAM.th} and Proposition~\ref{pred}, one checks that in Theorem~\ref{thth} the varieties with reductive automorphism groups are precisely $\PP^3$, $\PP^2\times\PP^1$, $\PP^1\times\PP^1\times\PP^1$, $\PP(a,a,b,b)$ with $1<a<b$, $(a,b)=1$,  and $\PP(a,a,b,b)[d]$ with $1\le a\le b$, $(a,b)=1$, ${d\ge 2}$. 

It also follows from Proposition~\ref{pred} that the automorphism group of the well-formed weighted projective space $X=\PP(d_0,\ldots,d_n)$ is reductive if and only if the subset $V(\Sigma)$ obtained from the multiset $\AAA(\Sigma)=\{d_0,\ldots,d_n\}$ is minimal in $\ZZ$.


\section{Appendix. Cox rings and root subgroups}
\label{App}


In this appendix we recall basic facts on Cox rings and related quotient presentations of toric varieties, and prove Theorems~\ref{teo1} and~\ref{teo2}.

Recall that the \emph{Cox ring} $C(X)$ of a non-degenerate toric variety $X$ is the polynomial ring $\KK[x_1,\ldots,x_m]$, where $m$ is the number of $T$-invariant prime divisors on $X$. The ring $C(X)$ is graded by the divisor class group $\Cl(X)$ with $\deg(x_i)=[D_i]$, where $[D_i]$
are the classes of the $T$-invariant prime divisors $D_i$ on $X$.  The spectrum $\overline{X}=\Spec C(X)$ is the affine space $\AA^m$, called the \emph{total coordinate space} of the toric variety $X$.  

Further, one defines the \emph{Cox sheaf} $\CCC_X$, a sheaf of $\Cl(X)$-graded rings on $X$, whose ring of global sections is $C(X)$. The relative spectrum $\widehat{X}=\Spec_X\CCC_X$ with the natural projection $q\colon\widehat{X}\to X$ is called the \emph{characteristic space} of $X$. The morphism $q$ is a good quotient with respect to the action of the quasitorus $Q:=\Spec \KK[\Cl(X)]$ defined by the $\Cl(X)$-grading on $\CCC_X$:  
$$
\begin{tikzcd}
\widehat{X} \ar{d}{q} \ar[hookrightarrow]{r} &
\overline{X} = \AA^m
\\
X
\end{tikzcd}
$$
The characteristic space can be embedded as an open $Q$-invariant subset $\widehat{X}\subseteq\overline{X}$, whose complement $Z$ is the zero set of the so-called \emph{irrelevant ideal} $\III(X)$ in $C(X)$. The subset  $Z$ is a union of some coordinate planes of codimension at least~2 in~$\AA^m$. More precisely, for a subset $I\subseteq\{1,\ldots,m\}$ the coordinate plane $\{x_i=0,\, i\in I\}$ is contained in $Z$ if and only if $\cone(p_i, i\in I)$ is not contained in a cone of $\Sigma$.  We have $\overline{X}=\widehat{X}$ if and only if $X$ is affine.  Notice that the linear action of the torus $\TT$ of all diagonal matrices on $\AA^m$ preserves the subset $\widehat{X}$. Further, the torus $\TT/Q$ acts on the quotient space $X$ and is identified with  the acting torus $T$. We refer to~\cite{Cox} and \cite[Sections~1.6,~2.1]{ADHL} for details on this construction. 

\begin{proof}[Proof of Theorem~\ref{teo1}.] 
It is known that any $\Ga$-action on $X$ can be lifted uniquely to a $\Ga$-action on $\overline{X}$. The lifting preserves the open subset $\widehat{X}$ and commutes with the action of the quasitorus $Q$; see~\cite[Theorem~4.2.3.2]{ADHL}. The lifting of a $T$-root subgroup $H$ of $\Aut(X)$ is a $\TT$-root subgroup $\HH$ of $\Aut(\overline{X})$.  Conversely, any $\TT$-root subgroup $\HH$ of $\Aut(\overline{X})$, which commutes with $Q$ and preserves the subset $\widehat{X}$, descends to a $T$-root subgroup $H$ of $\Aut(X)$. 

It is well known that $\TT$-root subgroups $\HH$ on $\AA^m$ are in bijection with homogeneous locally nilpotent derivations 
$$ 
\partial= x_1^{a_1}\ldots \widehat{x}_i \ldots x_m^{a_m}\frac{\partial}{\partial x_i} 
$$
of the polynomial ring $\KK[x_1,\ldots,x_m]$ equipped with the standard $\ZZ^m$-grading; cf.~\cite[Section~1.5.1]{Fr}. Such a subgroup $\HH$ commutes with $Q$ if and only if the derivation $\partial$ has degree zero with respect to the $\Cl(X)$-grading on $\KK[x_1,\ldots,x_m]$. 
This condition is equivalent to  $[D_i]=\sum_{j\ne i} a_j[D_j]$. In turn, this is equivalent to existence of a vector $e=e(\HH)\in M$ with ${\langle p_i, e\rangle=-1}$ and ${\langle p_j, e\rangle=a_j\ge 0}$ for all $j\ne i$; 
see~\cite[Section~4.1]{CLS} or \cite[Lemma~3.3, Proposition~3.4]{Fu}. This is precisely condition (R1) in Definition~\ref{defdr}. 

Now we check when the subgroup $\HH$ corresponding to $\partial$ preserves the open subset $\widehat{X}$ in $\AA^m$. The subgroup $\HH$ fixes all coordinates $x_1,\ldots,x_m$ except $x_i$ and sends $x_i$ to ${x_i+\alpha x_1^{a_1}\ldots\widehat{x}_i\ldots x_m^{a_m}}$ with $\alpha\in\KK$.  

Let $I(\HH)\subseteq\{1,\ldots,\widehat{i},\ldots,m\}$ be the set of those $j$ where $a_j>0$.  Equivalently, for $j \ne i$,  $j\notin I(\HH)$ if and only if $\langle p_j, e\rangle=0$.  Observe that $\TT$-orbits on $\AA^m$ are of the form $O(J)$, where $J\subseteq\{1,\ldots,m\}$ and $O(J)$ is the set of points $(x_1,\ldots,x_m)$ with $x_j\ne 0$ if and only if~$j\in J$. 

It is clear that $I(\HH)\nsubseteq J$ if and only if the $\TT$-orbit $O(J)$ consists of $\HH$-fixed points. The remaining $\TT$-orbits form pairs $(O(J'), O(J''))$, where $I(\HH)\cup\{i\}\subseteq J'$ and $J''=J'\setminus\{i\}$,  and the subgroup $\HH$ connects these two $\TT$-orbits.  Moreover, the closure of the $\TT$-orbit $O(J'')$ is a $\TT$-invariant prime divisor in the closure of the $\TT$-orbit $O(J')$. In particular, an open $\TT$-invariant subset $U$ in $\AA^m$ is $\HH$-invariant if and only if, whenever $U$ contains $O(J')$ from a pair $(O(J'), O(J''))$, it also contains $O(J'')$. 

Notice that the $\TT$-orbit $O(J)$ is contained in the characteristic space $\widehat{X}$ if and only if $\cone(p_k, k\notin J)$ is contained in a cone from $\Sigma$. So the subset $\widehat{X}$  is $\HH$-invariant if and only if  for any cone $\sigma\in\Sigma$ with $e|_{\sigma}\equiv 0$, where $e=e(\HH)$, we have $\cone(\sigma, p_i)\in \Sigma$. This is precisely condition (R2) in Definition~\ref{defdr}.  

We conclude that a $\TT$-root subgroup $\HH$ on $\AA^m$ descends to a $T$-root subgroup $H$ on $X$ if and only if the corresponding vector $e=e(\HH)$ is a Demazure root of the fan $\Sigma$. This completes the proof of Theorem~\ref{teo1}. 
\end{proof}

\begin{proof}[Proof of Theorem~\ref{teo2}.] 
We use notation from the proof of Theorem~\ref{teo1}. Denote by $\HH_e$ the lifting to $\AA^m$ of the $T$-root subgroup $H_e$. We are going
to deduce the desired results on the action of $H_e$ on $X$ from properties of the lifted action of $\HH_e$ on $\AA^m$.  

As we have observed in the proof of Theorem~\ref{teo1}, a $\TT$-orbit $O(J)$ on $\AA^m$ consists of $\HH_e$-fixed points if and only if $I(\HH_e)\nsubseteq J$. The orbit $O(J)$ projects to the $T$-orbit $O(\sigma)$ on $X$, where $\sigma\in\Sigma$ is the unique cone whose relative interior meets the relative interior of $\cone(p_i, i \notin J)$. The condition $I(\HH_e)\nsubseteq J$ means that there is $p_k\in\sigma$ such that $\langle p_k, e\rangle>0$. We conclude that any $T$-orbit 
$O(\sigma)$, such that $e|_{\sigma}$ takes a positive value, consists of $H_e$-fixed points.

All other $T$-orbits $O(\sigma)$ are of two types, namely, with $e|_{\sigma}=0$ and with $\sigma=\cone(\sigma',\rho_e)$, where $e|_{\sigma'}=0$. 
By Definition~\ref{defdr}, such orbits are divided into pairs $(O(\sigma'), O(\sigma''))$ with $e|_{\sigma'}=0$ and $\sigma''=\cone(\sigma',\rho_e)$. 
The preimages of these orbits under $q$ contain the $\TT$-orbits $O(J')$ and $O(J'')$, respectively, where $I(\HH_e)\cup\{i\}\subseteq J'$ and
$J''=J'\setminus\{i\}$. 

The subgroup $\HH_e$ acts on $\AA^m$ as
$$
x_i\mapsto x_i+\alpha x_1^{a_1}\ldots\widehat{x}_i\ldots x_m^{a_m},\quad\alpha\in\KK, \quad x_j\mapsto x_j, \quad j\ne i,
$$
where $p_e=p_i$. This shows that any $\HH_e$-orbit on $O(J')\cup O(J'')$ intersects $O(J'')$ in one point. A lifting of the one-parameter subgroup $\lambda_e$ in $T$ to $\TT$ is the $i$th coordinate subtorus $\TT_i$ acting on $\AA^m$ by $x_i\mapsto tx_i$ with $t\in\KK^{\times}$ and $x_j\mapsto x_j$ for $j\ne i$. It follows that the orbits of $\TT_i$ on $O(J')$ are precisely the intersections of $\HH_e$-orbits with $O(J')$. Applying the projection~$q$, we conclude that each $H_e$-orbit in $O(\sigma')\cup O(\sigma'')$ intersects $O(\sigma')$ in a one-dimensional orbit of the subtorus $\lambda_e$, while the same $H_e$-orbit meets $O(\sigma'')$ in a single point. Theorem~\ref{teo2} is proved. 
\end{proof} 

\begin{remark}
It follows from Proposition~\ref{codim2} that a non-degenerate toric variety $X$ is homogeneous in codimension one if and only if the group $\GG(X)$ generated by the liftings $\HH_e$ of all root subgroups $H_e$ in $\Aut(X)$ acts on the total coordinate space $\AA^m$ with an open orbit~$\OO$. Indeed, if the set 
$\RRR_{\rho_i}(\Sigma)$ is empty for some $\rho_i\in\Sigma(1)$, the group $\GG(X)$ does not change the coordinate $x_i$ on $\AA^m$, so it has no open orbit. 
Conversely, the one-dimensional coordinate subtori $\TT_1,\ldots,\TT_m$ generate the torus $\TT$. If all the sets $\RRR_{\rho_i}(\Sigma)$ are non-empty, we can apply shifts along the orbits of  $\TT_1,\ldots,\TT_m$ in the open $\TT$-orbit in $\AA^m$ (see the proof of Theorem~\ref{teo2}) and conclude that the open $\TT$-orbit in $\AA^m$ is contained in one $\GG(X)$-orbit $\OO$. 

Observe that if a group $\GG$ generated by some $\Ga$-subgroups acts on $\AA^m$ with an open orbit $\OO$, then this orbit is big, i.e., the complement of $\OO$ contains no divisor. Indeed, if $D$ is a prime divisor in the complement,  then $D$ is $\GG$-invariant. But any divisor on $\AA^m$ is principal, so $D=\text{div}(f)$ for some polynomial $f$ on $\AA^m$. Then $f$ is semi-invariant for any $\Ga$ -subgroup in $\GG$; see~\cite[Theorem~3.1]{PV}. Since the group $\Ga$ has no non-trivial character, the function $f$ is $\GG$-invariant, a contradiction with the existence of an open orbit. 

This observation implies that a non-degenerate toric variety $X$ is homogeneous in codimension one if and only if the group $G(X)$ generated by all root subgroups $H_e$ in $\Aut(X)$ acts on $X$ with a big open orbit. At the same time, one can easily construct projective toric surfaces $X$ such that the group $G(X)$ acts on $X$ with an open orbit that is not big.
\end{remark}


{}
\end{document}